\documentclass{article}
\usepackage{arxiv}
\usepackage[utf8]{inputenc} %
\usepackage[T1]{fontenc}    %
\usepackage{url}            %
\usepackage{booktabs}       %
\usepackage{amsfonts}       %
\usepackage{nicefrac}       %
\usepackage{microtype}      %
\usepackage{lipsum}		%
\usepackage{graphicx}
\usepackage{url}
\usepackage{graphicx}
\usepackage{mathtools}
\usepackage{amsfonts}
\usepackage{indentfirst}
\usepackage{amsthm}
\usepackage{subcaption}
\usepackage{diagbox}
\usepackage{multirow}
\usepackage{booktabs}
\usepackage{float}
\usepackage{hyperref}
\usepackage[
backend=biber,
style=numeric
]{biblatex}
\newtheorem{theorem}{Theorem}
\newtheorem{lemma}{Lemma}
\begin{document}
\title{Deep Euler method: solving ODEs by approximating the local truncation error of the Euler method}

\author{
	Xing Shen \\
	School of Mathematical Sciences\\
	Zhejiang University\\
	Hangzhou, Zhejiang, China\\
	\texttt{shenxingsx@zju.edu.cn}\\
	\And
	Xiaoliang Cheng \\
	School of Mathematical Sciences\\
	Zhejiang University\\
	Hangzhou, Zhejiang, China\\
	\texttt{xiaoliangcheng@zju.edu.cn}\\
	\And
	Kewei Liang\thanks{} \\
	School of Mathematical Sciences\\
	Zhejiang University\\
	Hangzhou, Zhejiang, China\\
	\texttt{matlkw@zju.edu.cn}\\
}

\maketitle

\begin{abstract}
\qquad In this paper, we propose a deep learning-based method, deep Euler method (DEM) to solve ordinary differential equations.  DEM significantly improves the accuracy of the Euler method by approximating the local truncation error with deep neural networks which could obtain a high precision solution with a large step size. The deep neural network in DEM is mesh-free during training and shows good generalization in unmeasured regions. DEM could be easily combined with other schemes of numerical methods, such as Runge-Kutta method to obtain better solutions. 
Furthermore, %
the error bound and stability of DEM is discussed.

\end{abstract}


\section{Introduction}
		Many problems in science and engineering can be modeled into a set of ordinary differential equations (ODEs)
		$$
			G(x,y,y^{\prime},y^{\prime\prime},\cdots) = 0, \quad x\in [a,b]\subset \mathbb{R}.
		$$
		In most cases, it can not be easy to obtain the analytic solution  
		and so one must typically rely on a numerical scheme to accurately approximate the solution. 		
		The important issues confronting the numerical study appear in the initial value problems since higher-order ODEs can be converted into the system of the first-order ODEs.
		Basic methods for initial value problems are the extremely popular Euler method or the Runge-Kutta method. 
		However, numerical methods have often to balance the discretization step size and computation time. %
		Furthermore, the class of stiff ordinary differential equations may still present a more serious challenge to numerical computation. 
In recent years, there has been a growing interest in solving the differential equations and the inverse problems by deep learning. The works include  numerical solutions of ODEs and PDEs (\cite{rudy2019deep}, \cite{raissi2018deep}, \cite{sun2019neupde}, \cite{farimani2017deep}), recovery of the involving systems (\cite{both2019deepmod}, \cite{khoo2017solving},  \cite{khoo2019switchnet}), overcoming the curse of dimension of high-dimensional PDEs (\cite{han2017overcoming},  \cite{hutzenthaler2018overcoming}), uncertainty quantification (\cite{tripathy2018deep},  \cite{winovich2019convpde}) etc. Besides, several works have focused on the combination of traditional numerical methods and deep neural networks. (\cite{sirignano2018dgm}) proposed a merger of Galerkin methods and deep neural networks (DNNs) to solve high-dimensional partial differential equations (PDEs). They trained   %
DNNs to satisfy the differential operator, initial condition, and boundary conditions. 
(\cite{raissi2019physics}) introduced  
physics-informed neural networks (PINNs), which is a deep learning framework %
for the synergistic combination of mathematical models and data. 
Following the physical laws of the control dynamics system, PINNs can deduce the solution of PDE and obtain the surrogate model. 
(\cite{weinan2018deep}) presented a deep Ritz method for the numerical solution of variational problems based on the Ritz method.
(\cite{he2018relu}) theoretically analyzed the relationship between DNN and finite element method(FEM). They explored the ReLU DNN representation of a continuous piecewise linear basis function in the finite element method.
(\cite{long2019pde}) proposed PDE-Net to predict the dynamics of complex systems. 
The underlying PDEs can be discovered from the observation data by establishing the connections between convolution kernels in CNNs and differential operators. 
Based on the integral form of the underlying dynamical system, (\cite{qin2019data}) considered ResNet block as a one-step method and recurrent ResNet and recursive ResNet as multi-step methods.
(\cite{wu2020data}) approximated the evolution operator by a residual network to solve and recover unknown time-dependent PDEs.
(\cite{raissi2018multistep}) blended the multi-step schemes with deep neural networks to identify and forecast nonlinear dynamical systems from data. 
(\cite{regazzoni2019machine}) proposed neural networks based Model Order Reduction technique to solve dynamical systems arising from differential equations.
(\cite{wang2019learning}) used reinforcement learning to empower Weighted Essentially Non-Oscillatory Schemes(WENO) for solving 1D scalar conservation laws. 

It is well known that the forward Euler method is very easy to implement but it can't give accurate solutions. 
The main reason is that the Euler method has only one order approximation accuracy, which requires a very small step size for any meaningful result. 
This makes the Euler method rarely used in practical applications and motivates us to propose a new Euler method combined with DNNs. 
We call the new method as deep Euler method (DEM). 
DEM only has the most general structure of a fully connected neural network, without any special designs in its structure, such as residual connections. 
As with some other deep neural network models, DEM also learns its representation using supervised pre-training. 
After the neural network gets trained satisfactorily, we post-process it to predict the solution of the ODE. 
The key difference is that in DEM, we explicitly capture information of the local truncation error of the Euler method instead of directly approaching the solution of the ODE. %
DEM has achieved state-of-the-art performance in solving ODEs, which is much better than the conventional numerical method, especially than the classical Euler method. 
This success can be attributed to the ability of the deep neural network in learning very strong hierarchical nonlinear representation. In particular, breakthroughs in supervised learning training are essential for deep neural networks to effectively and robustly predict.

The paper is organized as follows. %
We introduce the main idea of DEM in section 2 and  %
give theoretical results of DEM in section 3. Based on DEM, we also derive other schemes of the single-step method for solving ODEs in section 4.
In Section 5, numerical examples are given to demonstrate the capability and effectiveness of DEM. Finally, we conclude the paper in section 6.

\section{Deep Euler Method}

\subsection{Formulation} %

Considering the following initial value ordinary differential equation:
\begin{equation}\label{eq:1order_original_ODE}
\left\{
\begin{array}{l}
	\frac{d y}{d x}=f(x, y), \quad x \in I = [a, b], \\ %
	y(a)= c, 
\end{array}
\right.
\end{equation}
where the solution $y(x) : I \rightarrow \Omega \subset \mathbb{R}^n$ and $f$ satisfies the Lipschitz condition in $y$, i.e.,
$$\| f(x,y_1) -f(x,y_2)\| < L\|y_1-y_2\|.$$
We introduce the discretization mesh (or sampling points) in $x$, 
	$$
		a = x_0 < x_1 < \cdots < x_{M} = b.
	$$
	Let $h_{m} = x_{m+1}  - x_{m}$ be the mesh size and $y_{m}$ be the numerical approximation of $y(x_{m})$. 
	The forward Euler method for (\ref{eq:1order_original_ODE}) is
	\begin{equation}
	\left\{
		\begin{array}{l}
			y_{m} = y_{m-1} + h_{m-1} f(x_{m-1},y_{m-1}), \quad m=1,\cdots,M \\ 
			y_0= c.
		\end{array}
		\right.
	\end{equation}	
	Note that in most cases, we always adopt the uniform mesh for the forward Euler method, 
	$h = h_m $ and $x_m = a + mh$, $m=0,1,\cdots,M-1$.  
The local truncation error and the global error are defined as %
\begin{equation}\label{eq:local_truncation_error_define}
	R_m = y(x_{m+1}) - y(x_m) - hf(x_m,y(x_m)) = \int_{x_m}^{x_{m+1}}f(x,y(s)) ds - h f(x_m,y(x_m)).
\end{equation}
and 
$$
	e = |y(x_m) - y_m |,
$$
respectively. 
It is well known that %
$R_m = \mathcal O (h^2)$  and %
$e = %
\mathcal O (h)$ (\cite{Stig2008book}). %
To obtain higher accuracy than the Euler method, a direct scheme is to separate a part of $R_m$ to improve the Euler step $y_{m+1}$, so as to reduce the local truncation error and the global error of Euler method. To this end, we introduce a feedforward neural network in DEM that infers the update of an Euler step. As universal approximators, multilayer fully connected feedforward neural networks can approximate any continuous function arbitrarily (\cite{leshno1993multilayer}).
From (\ref{eq:local_truncation_error_define}), we could consider $R_m$ as a continuous function of variables $x_m,x_{m+1},y_m$. Thus, we utilize the fully connected neural network trained with enough measured data to approximate $\frac{1}{h_m^2}R_m$.

Let $\mathcal{N}(x_i,x_j,y_i;\theta) : \mathbb{R}^{n+2} \rightarrow \mathbb{R}^{n}$ be the nonlinear operator defined by a multilayer fully connected %
neural network. The parameter $\theta$ includes all the weights and the biases in the neural network.
DEM for (\ref{eq:1order_original_ODE}) can be written as
\begin{equation}\label{eq:ODE_DEM_formula}
	\left\{
	\begin{array}{l}
		y_{m+1} = y_m + h_m f(x_m,y_m) + h_m^2 \mathcal{N}(x_m,x_{m+1},y_m;\theta), \quad m=0,\cdots, M-1,\ \\
		y_0 = c.
	\end{array}
	\right.
\end{equation}
Formula (\ref{eq:ODE_DEM_formula}) consists of Euler approximation and neural network approximation. The first part makes full use of the information of $f$ to express the linearity of ODE. The latter corrects the results of Euler approximation to obtain higher accuracy and express nonlinear features. %
Abstractly, %
${\mathcal {N}}$ %
can be thought of as a parametric function that learns how to represent the local truncation error so that their most salient characteristics can be reconstructed from its inputs and outputs. 
The output of ${\mathcal{N}}(x_{m},x_{m+1},y_{m};\theta)$ contains all features extracted in training to update the formula of the Euler method. 
Compared with using a neural network to approximate the solution of ODE directly, DEM 
separates the nonlinear part from the numerical scheme and then makes full use of neural networks to approximate the local truncation error of Euler method. 
Moreover, ${\mathcal {N}}$ has the same function as the nonlinear denoising process. This provides a very powerful and flexible method for solving ODEs because we can impose high order error correction and reduce the constrain of the step size $h$ in the Euler method. Hence, DEM can either improve the accuracy of the Euler method or speed up the computations of ODEs. 

There are underlying principles for designing neural network architecture. In fact, we have another design of neural network 
$\mathcal{N}(x_i,x_j;\theta) : \mathbb{R}^{2} \rightarrow \mathbb{R}^{n}$. The output of the neural network is still an approximation of the local truncation error, while the input only has $x_{m}$ and $x_{m+1}$. In this case, if %
$n >> 2$, the neural network becomes very difficult to train. 
Because the dimension of the output is much larger than that of the input, it is almost impossible for the neural network to predict the sophisticated target with such few features. 

\subsection{Details of DEM} 
DEM is only a standard multilayer fully connected neural network,
without any special designs in its structure, such as residual connections.
With the input ${\bf{x}} = (x_i,x_j,y_i) \in \mathbb{R}^{n+2}$, the neural network in DEM can be written as
$$
	\mathcal{N} ({\bf{x}} ;\theta) = L_{K} \circ \sigma \circ L_{K-1} \cdots  \sigma \circ L_{1} ({\bf{x}}).
$$
The nonlinear activation function $\sigma(t) = \max\{0,t\}$ is rectified linear units (ReLU)  function. The $k$-th hidden layer %
has the following form 
$$
	L_{k}(z) = \boldsymbol{W}_{k} z + \boldsymbol{b}_{k}, \quad 1\leq k\leq K,
$$
where the weight matrix $\boldsymbol{W}_{k} \in \mathbb{R}^{p_k\times p_{k-1}}$,  
the bias $\boldsymbol{b}_{k}\in\mathbb{R}^{p_k}$, $p_{k}$ is the number of neurons in the $k$-th layer.

We assume that the measurement data is contaminated by noise, so that the training dataset $\mathbf D = \{(x_j,z_j)\}_{j=1}^{N}$ has the form  $z_j = y(x_j) + \delta_j$ and satisfies
$$
	\frac{1}{N} \sum_{j=0}^{N} \delta_j^2 \leq \delta^2,
$$
where the scalar $\delta$ is called the noise level.

For any pair of measurements $\{(x_i,z_i), (x_j,z_j)\}$ ($x_i<x_j$), we introduce the local truncation error function
\begin{equation}\label{eq:R_define}
	R(x_i,x_j,z_i,z_j) = \frac{1}{(\Delta x)^2} \left[ z_j - z_i - \Delta x f(x_i,z_i) \right],
\end{equation}
where $\Delta x = x_j - x_i$. With the following supervised loss
\begin{equation}
\label{eq:loss_define}
	J(\theta) = \frac{2}{N(N-1)} \sum_{1\leq i,j \leq N} \| {\mathcal{N}}(x_i,x_j,z_i;\theta) -  R(x_i,x_j,z_i,z_j) \|_{L^1},
\end{equation}
DEM learns to approximate the local truncation error of the Euler method. 
The coefficient comes from $ {\rm C}_N^2 = \frac{N(N-1)}{2}$ the number of pairs in dataset $D$.

Note that for any input pair $\{x_i, x_j\}$, each of the training captures the features in the local truncation error, which have close relations with %
$f$.  
Once features extractors corresponding to all pairs are trained and the strong hierarchical non-linear representations are generated, any new $y(x)$ ($x\neq x_{i}\in D$) is then represented by (\ref{eq:ODE_DEM_formula}). 
On the other hand, DEM is mesh-free because all training data can be generated randomly and  not necessary to locate at mesh points. 
Moreover, recalling the local truncation error of the Euler method is proportional to the square of the step size, we have $\frac{1}{h^2}R_m = \mathcal{O}(1)$. 
Hence, the neural network of DEM approaches a non-linear continuous function of  $\mathcal{O}(1)$, which is 
much easier than directly approximating the solution of the ODE.  
This makes DEM easier to train and requires fewer data in training. 

\section{Theoretical Analysis}

\subsection{Error Bound}

		\begin{lemma}\label{lemma}
			Assume that the trained neural network ${\cal N}$ satisfies
			$$\left|\mathcal{N}(x_m,x_{m+1},z_m;\theta)-R(x_m,x_{m+1},z_m,z_{m+1})\right|<\mathcal{O}(\eta).$$
			If $\delta<\eta$ and $h > \sqrt{\frac{\delta}{\eta}}$, then %
			$$
			\left|{\mathcal{N}}(x_{m},x_{m+1},y(x_m);\theta) - \dfrac{1}{h^{2}}R_{m}\right| < \mathcal{O}(\eta).
			$$
		\end{lemma}
		
		\begin{proof}
			From %
			Lemma 4 in \cite{xu2012robustness}, 
			we have the conclusion that the neural network in DEM is Lipschitz continuous. That is, 
			for any $\mathbf{x}_1,\mathbf{x}_2$, %
			$$
			|\mathcal{N}(\mathbf{x}_1;\theta)- \mathcal{N}(\mathbf{x}_2;\theta)| \leq L_{\mathcal{N}} \|\mathbf{x}_1-\mathbf{x}_2\| ,
			$$
			where $L_{\mathcal{N}}=\alpha^K \beta^K $,
			$\alpha = \max_{1 \leq k \leq K} \|W_k \|_{\infty}$, $\beta = \max _{a, b \in \mathbb{R}, a \neq b} \frac{|\sigma(a)-\sigma(b)|}{|a-b|}$,  $K$ is the number of layers of the neural network. 
			From
			\begin{eqnarray*} 
				\left| \mathcal{N}(x_m,x_{m+1},y(x_m);\theta)-\mathcal{N}(x_m,x_{m+1},z_m;\theta) \right| 
				& \leq &   L_{\mathcal{N}}|y(x_m)-z_m| \\
				& \leq & C  L_{\mathcal{N}} \delta < \mathcal{O}(\eta), \quad (C\  \text{is\  a\  constant}) 
			\end{eqnarray*}
			and %
			\begin{eqnarray*} 
				&&\left|R(x_m,x_{m+1},z_m,z_{m+1}) - \frac{1}{h^2}R_m\right| \\
				&=& \frac{1}{h^2}\left|[z_{m+1}-y(x_{m+1})]-[z_{m}-y(x_{m})]- \dfrac{}{}  h[f(x_m,z_m)-f(x_m,y(x_m))] \right|\\
				& \leq & \frac{2C\delta}{h^2}+ \frac{L\delta}{h}\\
				& \leq & \mathcal{O}(\eta),
			\end{eqnarray*}
			we have
			\begin{eqnarray*}
				\left|{\mathcal{N}}(x_{m},x_{m+1},y(x_m);\theta) - \dfrac{1}{h^{2}}R_{m}\right| &<&\left| \mathcal{N}(x_m,x_{m+1},y(x_m);\theta)-\mathcal{N}(x_m,x_{m+1},z_m;\theta) \right|\\
				&+& \left|R(x_m,x_{m+1},z_m,z_{m+1}) - \frac{1}{h^2}R_m\right|\\
				&+& \left|\mathcal{N}(x_m,x_{m+1},z_m;\theta)-R(x_m,x_{m+1},z_m,z_{m+1})\right|\\
				&\leq & \mathcal{O}(\eta)
			\end{eqnarray*}
		\end{proof}
		
		For each pair of measurements $\{(x_i,z_i),(x_j,z_j)\}$, we use (\ref{eq:R_define}) to construct the training samples of the neural network. %
		In the proof of Lemma \ref{lemma}, we have known that  $\left|R(x_m,x_{m+1},z_m,z_{m+1}) - \frac{1}{h^2}R_m\right| $ is smaller than a quantity which contains a factor $\frac{1}{h}$. This indicates the big $h$ is a good choice. Therefore, we will only select the measurement pair with the large $h=x_j-x_i$ to construct training samples.  

		\vskip 0.5cm
					
		\begin{theorem}\label{thm:local_error_DEM}
			Under the assumptions of Lemma \ref{lemma}, the local truncation error of DEM is $\mathcal O(\eta h^2)$ and the global truncation error is $\mathcal O(\eta h)$.
		\end{theorem}

\begin{proof}
From (\ref{eq:ODE_DEM_formula}), the local truncation error (LTE) of DEM is
\begin{eqnarray*}
	LTE &=& \left| y(x_{m+1}) - y(x_m) - hf(x_m,y(x_m)) - h^{2}{\mathcal{N}}(x_{m},x_{m+1},y(x_{m});\theta) \right| \\ 
	&\leq & h^2 \left|{\mathcal{N}}(x_{m},x_{m+1},y(x_{m});\theta) - \dfrac{1}{h^{2}}R_{m}\right| \\
	&<& \mathcal{O}(\eta h^2).
\end{eqnarray*}
Hence, we can also conclude that the global truncation error is $\mathcal O(\eta h)$.  
\end{proof}
Compared with the Euler method, the errors of DEM are reduced by $\eta$ times. 
The solution with high accuracy can be obtained. Besides, the size constrains of $h$ in the Euler method can be relaxed to speed up the computation of the solutions. For example, if the global error should be $O(10^{-6})$, the step size in the Euler method is at most $10^{-6}$. 
Starting from the initial $y(0)$, it takes $10^{6}$ Euler steps to get the solution $y(1)$.  
If $\eta = 10^{-4}$, the step size in DEM can be $10^{-2}$ and the number of steps can be reduced to $10^2$, which is much smaller than the Euler method.		

\subsection{Numerical stability}

\begin{theorem}
Under the assumptions of Theorem \ref{thm:local_error_DEM}, DEM is stable.
\end{theorem}
\begin{proof}
For the initial values $y_0$ and $z_0$ ($y_0\neq z_0$), DEM generates two approaches $\{y_m\}$ and $\{z_m\}$ with
\begin{equation}
\begin{array}{l}
{y_{m} = y_{m-1} + h_{m-1} f(x_{m-1},y_{m-1}) + h_{m-1}^2\mathcal{N}(x_{m-1},x_{m},y_{m-1};\theta)},\\

{z_{m} = z_{m-1} + h_{m-1} f(x_{m-1},z_{m-1}) + h_{m-1}^2\mathcal{N}(x_{m-1},x_{m},z_{m-1};\theta)}.\\
\end{array}
\end{equation}
Since $\sum_{m=0}^{M-1} h_m = b-a$ and $| \mathcal{N}(x_{m-1},x_{m},y_{m-1};\theta) - \mathcal{N}(x_{m-1},x_{m},z_{m-1};\theta)| \leq L_{\mathcal{N}} |y_{m-1}-z_{m-1}|$, we have
\begin{eqnarray*}
|y_{m} - z_{m}|   &\leq &    |y_{m-1} - z_{m-1}| + h_{m-1}|f(x_{m-1},y_{m-1})-f(x_{m-1},z_{m-1})|  \\
				&& \qquad + h_{m-1}^2 L_{\mathcal{N}} |y_{m-1}-z_{m-1}|\\
				&\leq & (1+h_{m-1}L+h_{m-1}^2L_{\mathcal{N}}) |y_{m-1} - z_{m-1}|\\
			&\leq & \prod_{n=0}^{m-1} (1+h_{n}L+h_{n}^2L_{\mathcal{N}})|y_{0} - z_{0}|\\
			&\leq & C|y(0)-z(0)|
\end{eqnarray*}
where $C$ is constant.

\end{proof}
Considering the stiff equation $\frac{dy}{dx} = \lambda y$, where $\lambda<0$ and the initial value  $y(a) = c$. The stability domain of DEM is $\{ h\in \mathbb{C} | \space |1+ h\lambda + h^2L_{\mathcal{N}}|\leq 1\}$, while the Euler method is $\{ h\in \mathbb{C} | \space |1+ h\lambda|\leq 1\}$. 
Although it can not be proved theoretically that the former must be larger than the latter,  
DEM can use  a large step in numerical experiments. Under such a step size, the forward Euler method is certainly unstable. On the other hand side, a large stability domain can be obtained by adjusting $L_{\mathcal{N}}$. 
Recall that $L_{\mathcal{N}} = \alpha^k\beta^K$, where $\beta$ is determined by the activation function. If ReLU activation function is adopted then $\beta=1$. We can change the norm of weight matrix $\alpha$ by using the techniques of weight clipping (\cite{salimans2016weight}) and weight normalization (\cite{arjovsky2017wasserstein}), %
to adjust $L_{\mathcal{N}}$.
For example, when $\lambda=-5 $ and $L_{\mathcal{N}} = 6$, The stability domain of DEM is $0<h\leq \frac{5}{6}$, while the Euler method is $0<h\leq \frac{2}{5}$. Hence we can use a larger step size to solve the equation than the Euler method.

\section{Single Step methods}
Based on the idea of approximating the local truncation error with a deep neural network, DEM could be generalized to other linear single-step methods. 
For instance, Heun's method  
$$y_{m+1}=y_{m}+\frac{h}{2}\left[f\left(x_{m}, y_{m}\right)+f\left(x_{m+1}, y_{m}+h f\left(x_{m}, y_{m}\right)\right)\right]  $$
is a second-order Runge-Kutta method. We can also add a deep neural network in it and get
\begin{equation}\label{eq:Heun_method_NN_ODE}
	y_{m+1}=y_{m}+\frac{h}{2}\left[f\left(x_{m}, y_{m}\right)+f\left(x_{m+1}, y_{m}+h f\left(x_{m}, y_{m}\right)\right)\right] 
			+ h_m^3{\cal N}(x_m,x_{m+1},y_m;\theta).
\end{equation}
More generally,
a $p$ order single-step method of (1) can be written as:
\begin{equation}\label{eq:p_order_single_step_ODE}
\left\{
\begin{array}{l}{
y_{m+1} = y_m + h \psi(x_m,y_m,h) }, \\
{y_0 = c}.
\end{array}
\right.
\end{equation}
The local truncation error is 
$$R_m = y_{m+1} - y_m - h\psi(x_m,y_m,h) = \mathcal{O}(h^{p+1}).$$
The method (\ref{eq:p_order_single_step_ODE}) can also be modified as 
\begin{equation}\label{eq:p_order_NN_ODE}
		\left\{
		\begin{array}{l}
			y_{m+1} = y_m + h \psi(x_m,y_m,h)  %
				+ h^{p+1}{\cal N}(x_m,x_{m+1},y_m;\theta), \\
			{y_0 = c}.
		\end{array}
		\right.
\end{equation}
For (\ref{eq:Heun_method_NN_ODE}) and (\ref{eq:p_order_NN_ODE}), the local truncation errors are $\mathcal{O}(\eta h^3)$ and $\mathcal{O}(\eta h^{p+1})$, respectively. The corresponding global truncation errors are $\mathcal{O}(\eta h^2)$ and $\mathcal{O}(\eta h^p)$. 
\section{Numerical Example}
\subsection{Example 1}
Considering the following initial value problem:
\begin{equation}\label{eq:the_first_example}
\left\{
\begin{array}{l}
{\frac{dy}{dx} = \frac{3}{2}\frac{y}{x+1}+\sqrt{x+1}}, \space x\in[0,10]\\
{y(0) = 0}
\end{array}
\right.
\end{equation}
where the exact solution %
is $y =(x+1)^{3/2} \log(x+1)$.

At first, we highlight the performance of DEM with different step sizes with noise-free measured data. We generate 200 random noise-free measured data $\left\{ (t_i, y(t_i) )\right\}_{i=1}^{200}$,  %
by sampling from a uniform distribution ${\cal F}_{t} = U(0,5)$, 
then we %
train the deep neural network ${\cal N}$ by minimizing the loss function of (\ref{eq:loss_define}). %
All norms used in this paper are $\mathcal{L}_1$ norm.  
The neural network, with $8$ layers, $80$ neurons and the ReLU activation function in each layer,
is trained for 50 epochs, optimized with Adam. All the learning rate in this paper is set to be $5\times 10^{-3}$. 
The same neural network architecture is used in Deep Heun's method (\ref{eq:Heun_method_NN_ODE}) for comparison. 
Figure \ref{fig:1} shows the evolution of the trained ${\cal N}(x_{m}, x_{m+1},y_{m};\theta)$ in DEM and %
the local truncation error function $R(x_{m}, x_{m+1},y_{m},y_{m+1})$ in the Euler method.
The four different step sizes, that is $h = 0.01,$  $0.1$, $1.0$ and $2.0$  are displayed.  
Since it is trained in $(0,5)$, ${\cal N}(x_{m},x_{m+1},y_{m};\theta)$ almost coincides with $R(x_{m},x_{m+1},y_{m},y_{m+1})$ in $(0,5)$. 

	Figure \ref{fig:2} shows the exact solution and four approximations of (\ref{eq:the_first_example}).  The four different step sizes are also displayed. It is observed that only in the small step $h$ can the Euler method and the Huen obtain a more accurate approximation of the solution. We also note the fact that ${\cal N}$ is only trained in $(0,5)$.  However, in $(5,10)$,  DEM and DHM can get the accurate approximation of the solution even for the bigger step size $h=2.0$.  This indicates the efficient prediction of the deep neural network.
	
	In Table \ref{table:the_first_example},  %
we discuss the results of the comparison among four methods for the different step sizes. 
The Euler method, the Heun's method, DEM and DHM are adopted for solving (\ref{eq:the_first_example}). 
The first four columns of Table \ref{table:the_first_example} show the prediction errors 
between the exact solution and the approximation in $L_1$ norm, i.e. $e= \max_{m}|y(x_m) - y_m|$. 
Since the global truncation error of Deep Euler Method and Deep Heun's method are $\mathcal{O}(\eta h)$ and $\mathcal{O}(\eta h^2)$.  That is the reason that when $h\geq 1$, DEM and DHM get more accuracy than the Euler method and the Heun method. 
Our goal is to get the estimates of $\eta$. 
In view of Lemma \ref{lemma}, $\left| {\cal N}(x_{m},x_{m+1},z_{m};\theta) - \dfrac{1}{h^{2}}R_{m} \right| \approx \left|  {\cal N}(x_{m},x_{m+1},z_{m};\theta) - R(x_{m},x_{m+1},z_{m},z_{m+1})  \right| =   \mathcal{O}(\eta)$. 
In the column of $\varepsilon_{mean}$, we present the mean of the difference between ${\cal N}$ and $R$, which is $\varepsilon_{mean}=\dfrac{1}{M}\sum_{m} \left| {\cal N} - R \right| = \mathcal{O}(\eta)$.  
In the last column, we present the results of DEM (the fourth column) divided by the result of the Euler method (the second column), i.e., ${e_{DEM}}/{e_{Euler}} = \mathcal{O}(\eta)$. 
From the last two columns, we can conclude that $\eta = {\cal O} (0.001)$, which also indicates the efficiency of DEM. 

Table \ref{table:Example1_layers_neurals_points} %
		shows $\varepsilon_{mean}$ for different network architectures (the number of hidden layers and neurons) and the different number of %
		random measured points. We evolution the neural networks with %
		$h=0.1$. To avoid uncertainty during the training process, we simulate each case %
		ten times and take the average value of them. 
		It can be observed that the prediction accuracy increase with the number of measured points. 
		If the networks with too small layers and neurons per layer (such as $2$ layers and $20$ neurons per layer), it is not suitable in the case of  a small number of measured data since it has a high bias in the training region $[0,5]$ and a high variance in the testing region $(5,10]$.  
		If the networks with the  more layers and neurons per layer (such as $16$ layers and $160$ neurons per layer), it may be overfitted in the case of the small number of measured data since it has low bias and high variance. More measured data can reduce the variance. 
		We choose the networks with $8$ layers and $80$ neurons per layer for all experiments.  
		Because no matter in a small amount or a large number of measured data, it has low variance and low bias than others.
\begin{figure}[!h]
	\centering
	\begin{subfigure}{0.48\textwidth}
		\includegraphics[width=\linewidth]{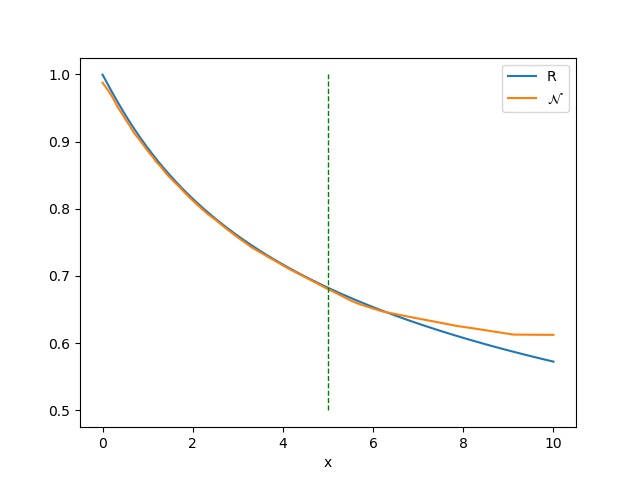}
			\caption{$h=0.01$}
	\end{subfigure}\hspace*{\fill}
	\begin{subfigure}{0.48\textwidth}
		\includegraphics[width=\linewidth]{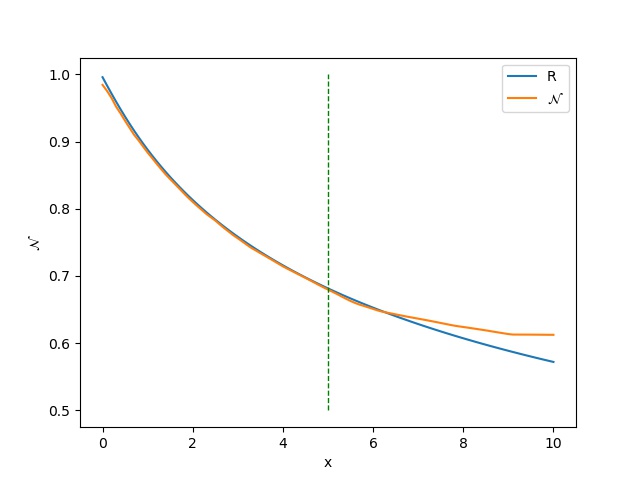}
						\caption{$h=0.1$}
	\end{subfigure}
	
	\medskip
	\begin{subfigure}{0.48\textwidth}
		\includegraphics[width=\linewidth]{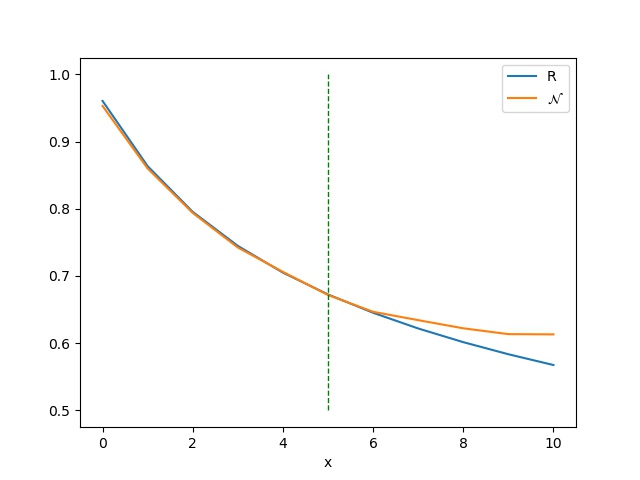}
						\caption{$h=1.0$}
	\end{subfigure}\hspace*{\fill}
	\begin{subfigure}{0.48\textwidth}
		\includegraphics[width=\linewidth]{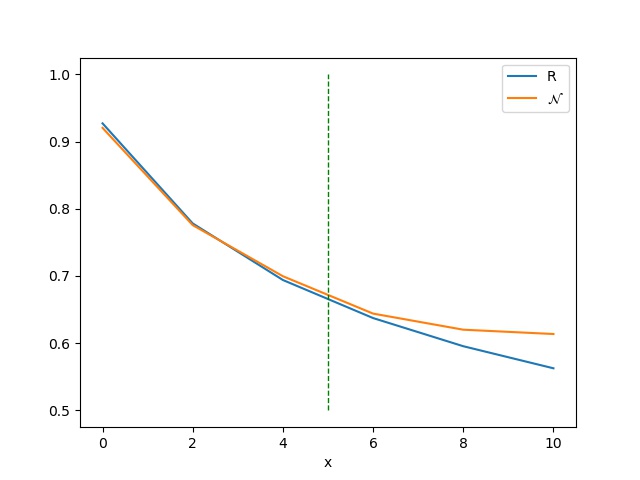}
					\caption{$h=2.0$}
	\end{subfigure}
	\caption{The evolution of the trained neural network ${\cal N}(x_{m}, x_{m+1},y_{m};\theta)$ in DEM and of the local truncation error function $R(x_{m}, x_{m+1},y_{m},y_{m+1})$ in the Euler method. The four different step sizes are $0.01$, $0.1$, $1.0$ and $2.0$. The red line is the result of ${\cal N}$. The blue line is the result of $R$.}
			\label{fig:1}
\end{figure}

\begin{table}[H]
	\centering
	\begin{minipage}{0.85\textwidth}
			\caption{ 
			The results of the comparison among four methods for different step sizes.  		
			Prediction errors between the exact solution and the approximation in $L_1$ norm are listed, i.e., $e= \max_{m}|y(x_m) - y_m|$.  		
			The column of $\varepsilon_{mean}$ is $\sum_{m} \left| {\cal N} - R \right|/M$, where $M$ is the number of the steps.  
			The last column is the result of DEM (the fourth column) divided by the result of the Euler method (the second column). 
		}
		\label{table:the_first_example}
	\end{minipage}
	{\small 
		\begin{tabular}{|c|c|c|c|c|c|c|}
			\hline
			step size & Euler method & Heun's method  & DEM & DHM & $\varepsilon_{mean}$ & ${e_{DEM}}/{e_{Euler}}$  \\ %
			\hline 
			0.01 & {0.42} & {0.0017}& {0.0014} & {0.000053} & {0.0086} & {0.0033}  \\
			\hline 
			{0.1} & {4.05} & {0.15} & {0.013}& {0.0051} & {0.0089}   & {0.0032}  \\
			\hline
			{1} & {28.42} & {8.10}& {0.073}& {0.32} & {0.012}  & {0.0026}   \\
			\hline 
			{2} & {43.16} & {18.78} & {0.083}& {1.03} & {0.016}   & {0.0019}  \\
			\hline 
		\end{tabular}
	}
\end{table}

\begin{figure}[H] 
	\centering
	\begin{subfigure}{0.48\textwidth}
		\includegraphics[width=\linewidth]{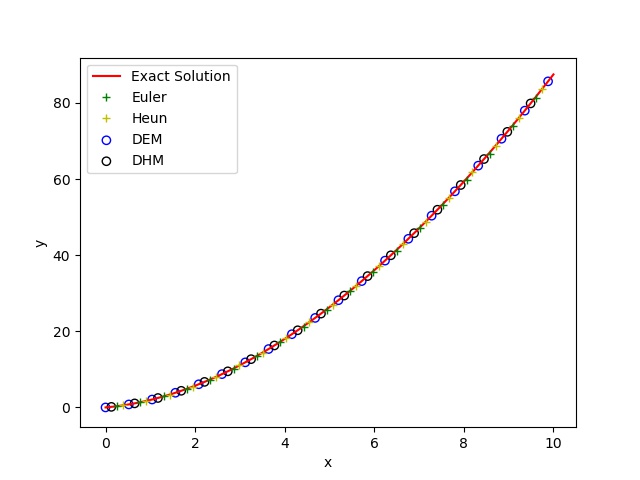}
			\caption{$h=0.01$}
	\end{subfigure}\hspace*{\fill}
	\begin{subfigure}{0.48\textwidth}
		\includegraphics[width=\linewidth]{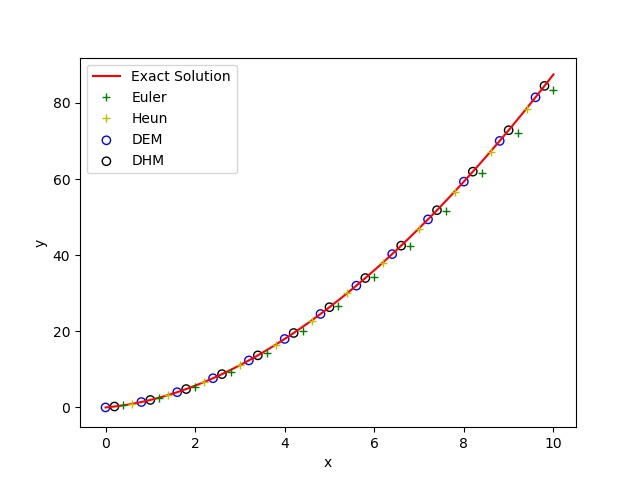}
						\caption{$h=0.1$}
	\end{subfigure}
	
	\medskip
	\begin{subfigure}{0.48\textwidth}
		\includegraphics[width=\linewidth]{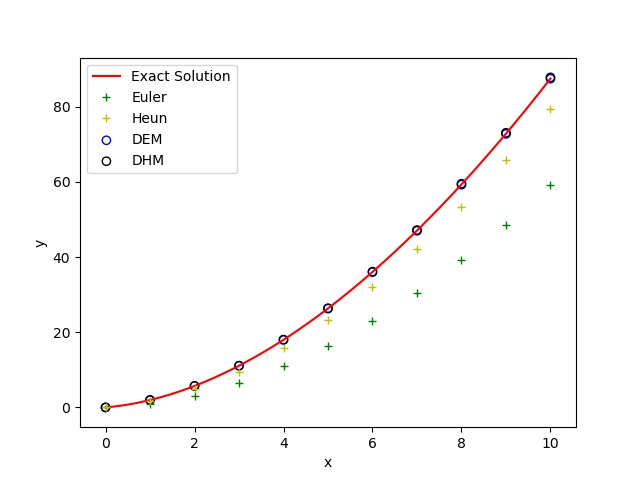}
					\caption{$h=1.0$}
	\end{subfigure}\hspace*{\fill}
	\begin{subfigure}{0.48\textwidth}
		\includegraphics[width=\linewidth]{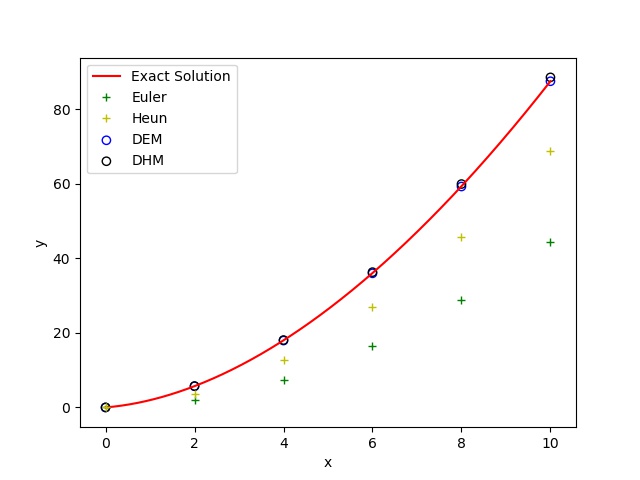}
					\caption{$h=2.0$}
	\end{subfigure}
		\caption{The exact solution and four approximations of (\ref{eq:the_first_example}). 		  
			The four different step sizes are $0.01$, $0.1$, $1.0$ and $2.0$. 
			The green plus is the result of the Euler method. The yellow plus is the result of the Heun method.
			The blue circle is the result of DEM. The black circle is the result of DHM.
		}
			\label{fig:2}
\end{figure}

\begin{table}[H]
	\centering
	\begin{minipage}{0.85\textwidth}
			\caption{ The results of $\varepsilon_{mean}$, i.e., the average error of between $\mathcal{N}$ and $R$, for the different number of hidden layers and neurons per layer, as well as the different number of measured points. The number on the left is the error in the training region $[0, 5]$, and the error in the test region $(5, 10]$ is on the right. All the step size is $h=0.1$. 
			}			\label{table:Example1_layers_neurals_points}
	\end{minipage}
	\begin{tabular}{|c|cc|cc|cc|cc|}
		\hline
		& \multicolumn{8}{c|}{layers \& neurons} \\
		\hline
		points %
		& \multicolumn{2}{|c|}{2 $\times$ 20}  & \multicolumn{2}{|c|}{4$\times$ 40} & \multicolumn{2}{|c|}{8$\times$ 80} & \multicolumn{2}{|c|}{16 $\times$ 160} \\
		\hline
		10 & {0.21} &{0.69} &  {0.067} &{0.36}  &  {0.033}&{0.11} &  {0.071}&{0.17} \\
		\hline
		25 &  {0.20}&{0.81} &  {0.03}&{0.16}  &  {0.014}&{0.061} &  {0.075}&{0.16} \\
		\hline
		50 &  {0.081}&{0.41} &  {0.024}&{0.36}  &  {0.022}&{0.073} &  {0.049}&{0.12} \\
		\hline
		100 &  {0.017}&{0.21} &  {0.0093}&{0.14}  &  {0.011}&{0.045} &  {0.014}&{0.052} \\
		\hline
		200 &  {0.0096}&{0.28} &  {0.0056}&{0.080}  &  {0.0093}&{0.030} &  {0.0084}&{0.035} \\
		\hline
		500 &  {0.0066}&{0.22} &  {0.0024}&{0.072}  &  {0.0028}&{0.048} &  {0.0039}&{0.048} \\
		\hline
	\end{tabular}
\end{table}

\begin{figure}[H]
	\begin{subfigure}{0.48\textwidth}
		\includegraphics[width=\linewidth]{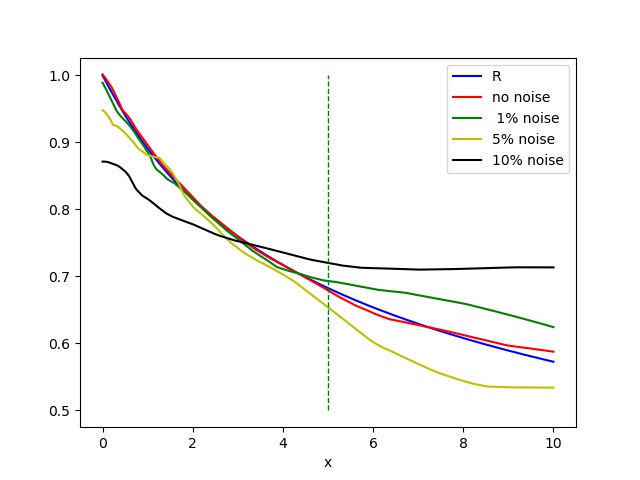}
		\caption{h=0.01} 
	\end{subfigure}\hspace*{\fill}
	\begin{subfigure}{0.48\textwidth}
		\includegraphics[width=\linewidth]{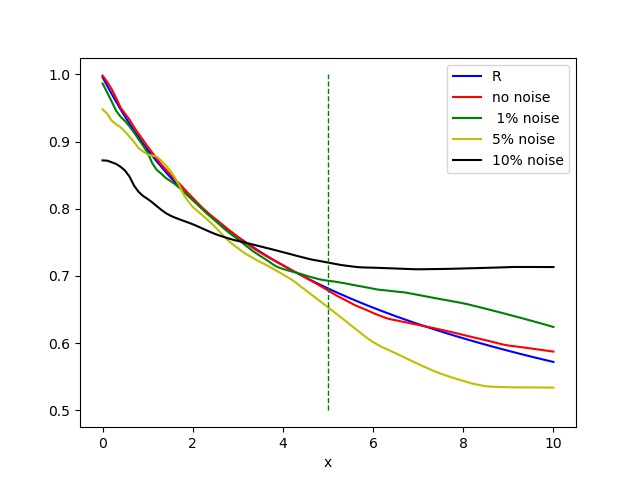}
		\caption{h=0.1} %
	\end{subfigure}
	
	\medskip
	\begin{subfigure}{0.48\textwidth}
		\includegraphics[width=\linewidth]{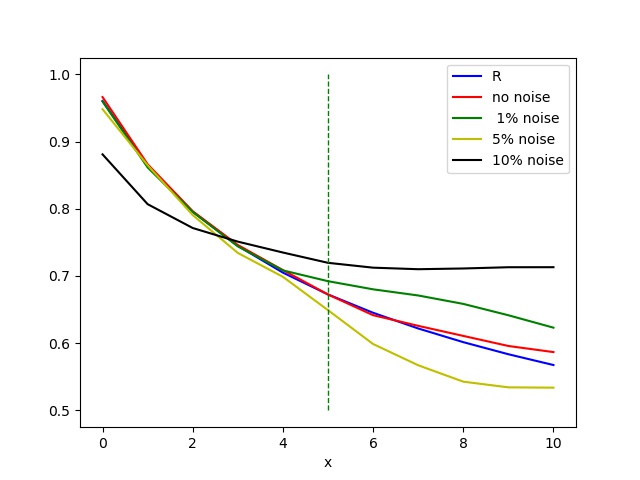}
		\caption{h=1.0} %
	\end{subfigure}\hspace*{\fill}
	\begin{subfigure}{0.48\textwidth}
		\includegraphics[width=\linewidth]{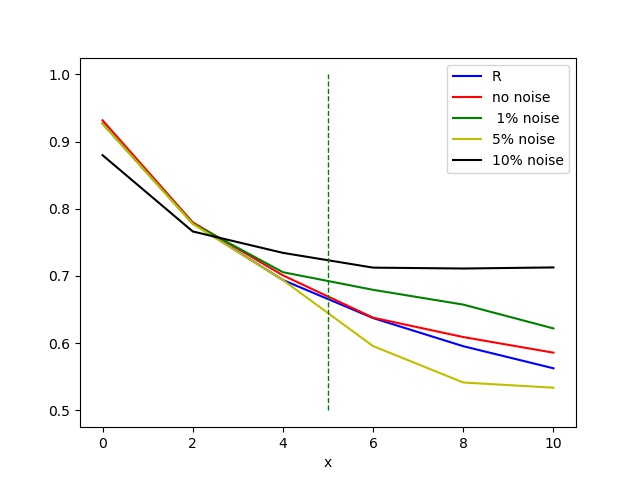}
		\caption{h=2.0} 
	\end{subfigure}
	\caption{
		For comparison, the DEMs and the local truncation error function $R$ are listed together. 
		Each ${\cal N}$ of DEM is trained %
		for one case of noise levels ($\delta = 0\%, 1\%,5\%,10\%$). 
		Four subfigures show the evolutions of the four ${\cal N}(x_{m}, x_{m+1},y_{m};\theta)$  and $R$ for $h=0.01$, $0.1$, $1.0$ and $2.0$, respectively.
	} 
	\label{fig:noise}
\end{figure}

Figure \ref{fig:noise} shows the DEMs and the local truncation error function $R$ together for comparison. Each ${\cal N}$ of DEM is trained for one case of noise levels ($\delta = 0\%, 1\%,5\%,10\%$). The four different step sizes, that is  $h= 0.01$, $0.1$, $1.0$ are displayed. It can be observed from the subfigures that the change of step size has little effect on the result. 
In fact, from Table \ref{table:noise}, we also noted that. 
When the step size $h$ changes from $0.01$ to $2.0$, $\varepsilon_{mean}$ under various noise levels ($\delta = 0, 1 \%, 5 \%, 10 \%$) are almost the same.
However,  obviously,  $e_{DEM}$ is the smallest in the case of the smallest $h$ and the lest noise level (noise-free).

\begin{table}[H]
	\centering
	\begin{minipage}{0.85\textwidth}
				\caption{The performance of DEM %
				for the different noise levels and the corresponding numerical solutions.
			}
			\label{table:noise}
	\end{minipage}
	\begin{tabular}{|c|cccc|cccc|}
		\hline
		& \multicolumn{4}{|c|}{$\epsilon_{mean}$}  &\multicolumn{4}{|c|}{$e_{_{DEM}}$} \\
		\hline
		\diagbox{ $h$ %
		}{$\delta $ %
		} & 0\% & 1\% & 5\% & 10\% & 0\% & 1\% & 5\% & 10\%  \\ 
		\hline
		0.01 & 0.005 & 0.02 & 0.04 & 0.07 & 0.001 & 0.001 & 0.01 & 0.02\\
		\hline
		0.1 & 0.005 & 0.02 & 0.03 & 0.07 & 0.01 & 0.01 & 0.01 & 0.02\\
		\hline
		0.5 & 0.005  & 0.02 & 0.03 & 0.07 & 0.06 & 0.09 & 0.35 & 0.58\\
		\hline
		1.0 & 0.006 & 0.03 & 0.03 & 0.07 & 0.12 & 0.27 & 0.50 & 0.71\\
		\hline
		2.0 & 0.01 & 0.03 & 0.02 & 0.07 & 0.24 & 0.55 & 0.45  & 0.50\\
		\hline
	\end{tabular}
\end{table}

\subsection{Example 2}
We now consider a system of first-order nonlinear differential equations, the Lotka-Volterra equation (\cite{lotka1925elements}),
\begin{equation}\label{eq:Example2_Lotka_Volterra}
		\left\{
		\begin{aligned}
				\frac{dy_1}{{dx}} &=\alpha y_1- \beta {y_1y_2}, \\ 
				\frac{{dy_2}}{{dx}}& = -\gamma y_2+ \delta {y_1y_2}.
		\end{aligned}
		\right.
\end{equation}
This equation
describes the dynamics of the populations of two species, one as a predator and the other as prey. 
In (\ref{eq:Example2_Lotka_Volterra}), %
$y_1$ and $y_2$ are the number of prey and predator, respectively. 
Let $\boldsymbol{y} = [y_1,y_2]^T$,  $x $ represent time and $\alpha,\beta,\gamma,\delta$ be the parameters describing the relationship of two species. 
In this work, we take $\alpha = \beta = \gamma =  \delta = 1 $, and the initial conditions $\boldsymbol{y} = [y_1(0),y_2(0)]^T=[2,1]^T.$

For comparison, the exact solution is gotten by the numerical  method of %
$RK45$ in scipy.integrate \cite{2020SciPy-NMeth} with $10^{-6}$ relative tolerances. 
We sample 1000 random points from a uniform distribution $\mathcal{F}_t = U(0,15)$ to construct the noise-free training dataset $\{(t_i,y(t_i)\}_1^{1000}$.  The deep neural network with $8$ hidden layers and $80$ neurons per layer is trained with super parameters the same as example 1.

\begin{figure}[H]
	\begin{subfigure}{0.48\textwidth}
		\includegraphics[width=\linewidth]{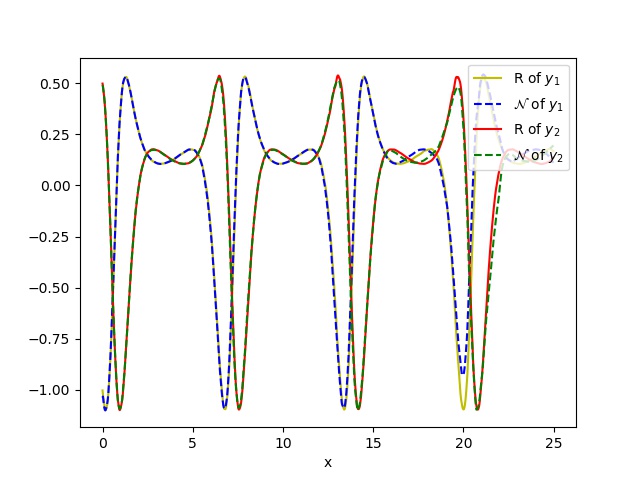}
		\caption{$h=0.01$} 
	\end{subfigure}\hspace*{\fill}
	\begin{subfigure}{0.48\textwidth}
		\includegraphics[width=\linewidth]{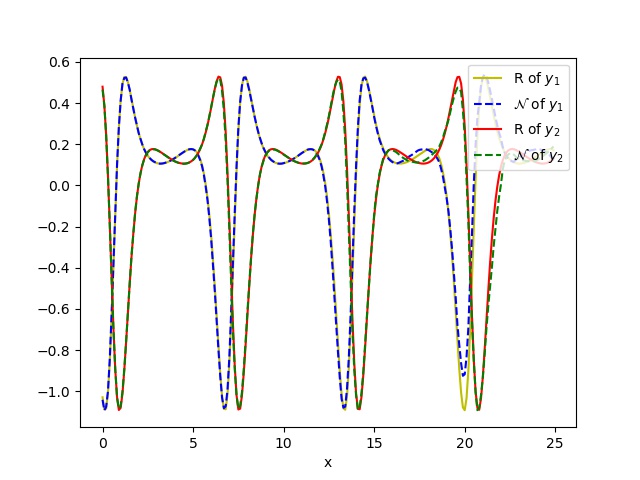}
		\caption{$h=0.1$} %
	\end{subfigure}
	
	\medskip
	\begin{subfigure}{0.48\textwidth}
		\includegraphics[width=\linewidth]{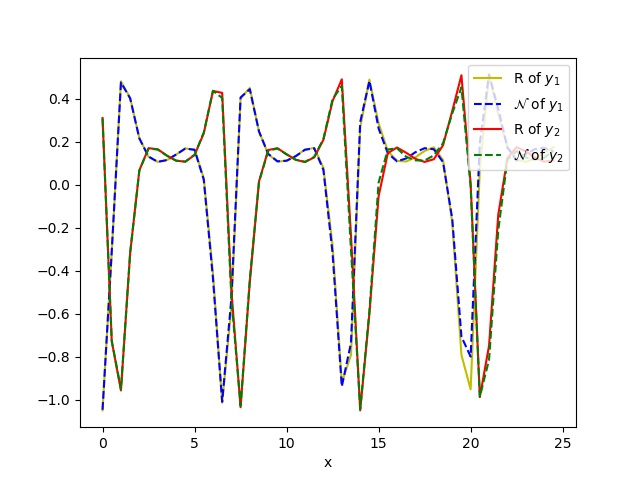}
		\caption{$h=0.5$} %
	\end{subfigure}\hspace*{\fill}
	\begin{subfigure}{0.48\textwidth}
		\includegraphics[width=\linewidth]{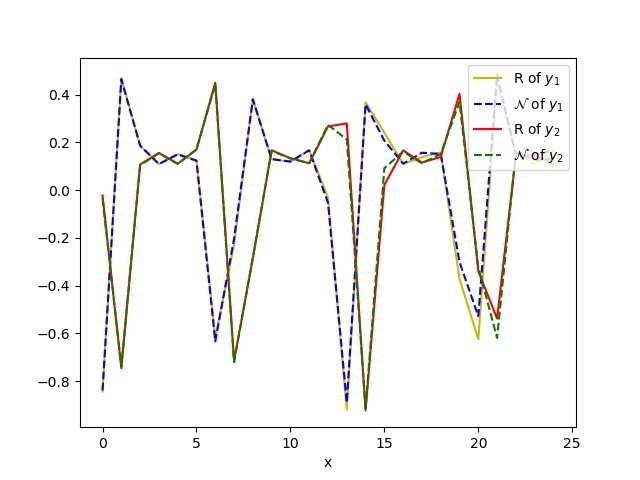}
		\caption{$h=1.0$} 
	\end{subfigure}
	
	\caption{The evolution of the trained neural networks $\mathcal{N}(x_m,x_{m+1},\boldsymbol{y_m})$ in DEM and the local truncation error function $R(x_m,x_{m+1},\boldsymbol{y_m},\boldsymbol{y_{m+1}})$ for different step sizes $h=0.01,0.1,0.5,1.0$ }
	\label{fig:4}
\end{figure}
\begin{figure}[H] 
	\begin{subfigure}{0.4\textwidth}
		\includegraphics[width=\linewidth]{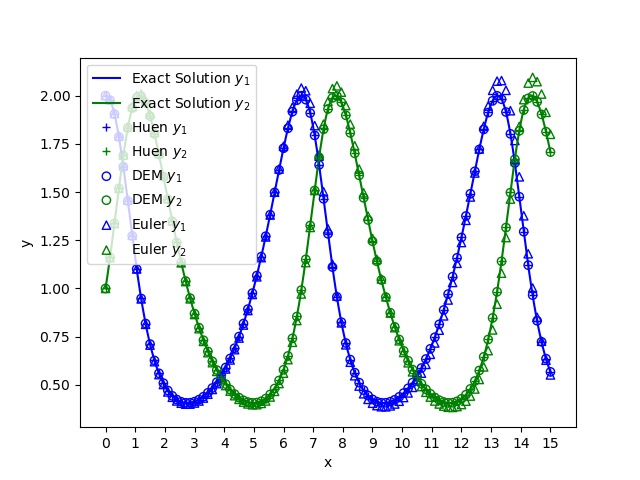}
		\caption{$h=0.01$ in [0,15]} 
	\end{subfigure}\hspace*{\fill}
	\begin{subfigure}{0.4\textwidth}
		\includegraphics[width=\linewidth]{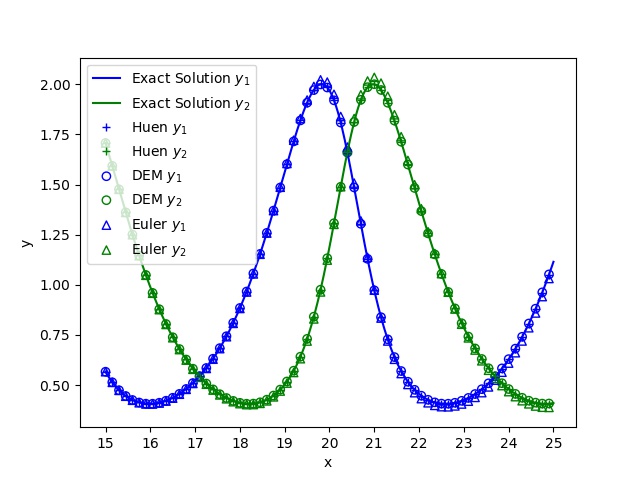}
		\caption{$h=0.01$ in [15,25]} %
	\end{subfigure}
	\medskip
	\begin{subfigure}{0.4\textwidth}
		\includegraphics[width=\linewidth]{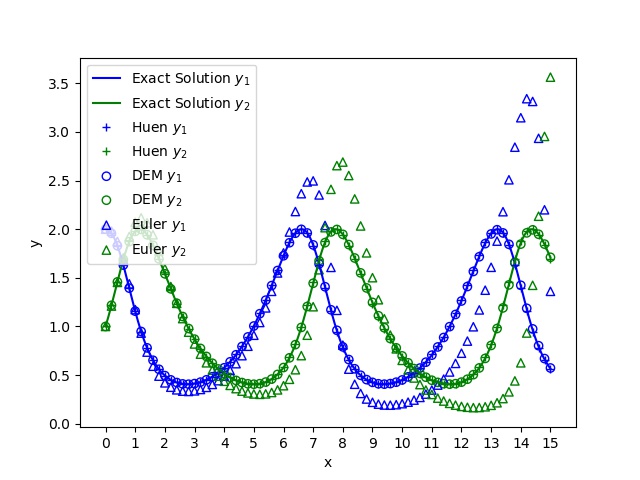}
		\caption{$h = 0.1$ in [0,15]} %
	\end{subfigure}\hspace*{\fill}
	\begin{subfigure}{0.4\textwidth}
		\includegraphics[width=\linewidth]{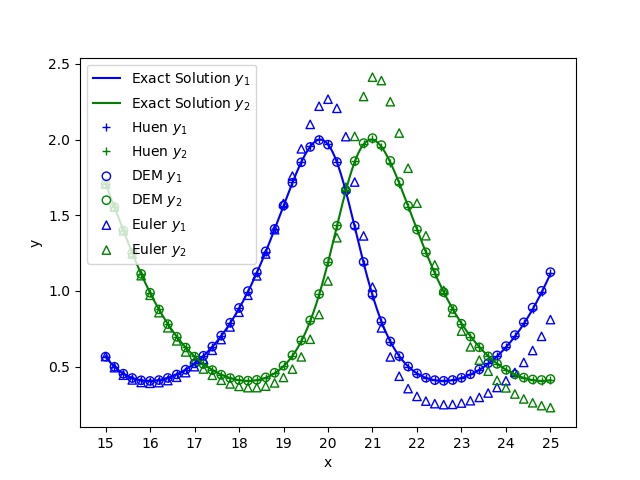}
		\caption{$h = 0.1$ in [15,25]} 
	\end{subfigure}
	\medskip
	\begin{subfigure}{0.4\textwidth}
		\includegraphics[width=\linewidth]{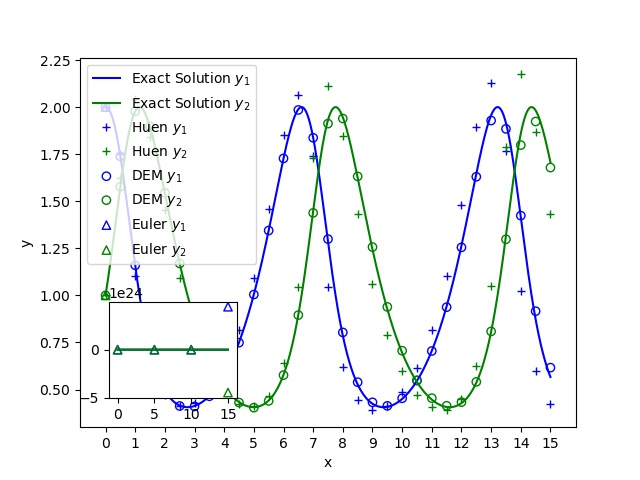}
		\caption{$h=0.5$ in [0,15]}
	\end{subfigure}\hspace*{\fill}
	\begin{subfigure}{0.4\textwidth}
		\includegraphics[width=\linewidth]{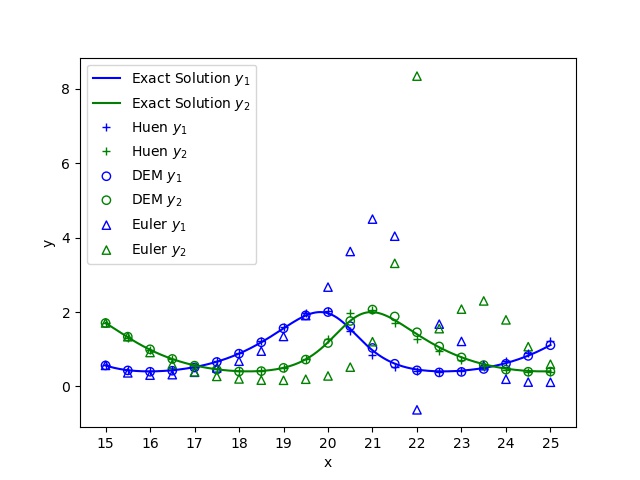}
		\caption{$h=0.5$ in [15,25]}
	\end{subfigure}
	\medskip
	\begin{subfigure}{0.4\textwidth}
		\includegraphics[width=\linewidth]{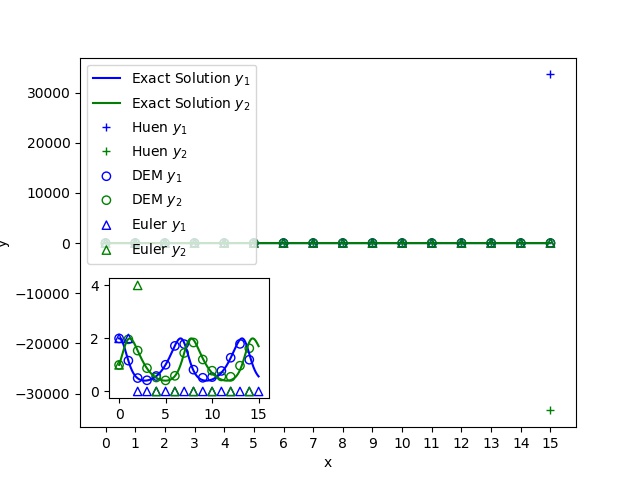}
		\caption{$h= 1.0$ in [0,15]} %
	\end{subfigure}\hspace*{\fill}
	\begin{subfigure}{0.4\textwidth}
		\includegraphics[width=\linewidth]{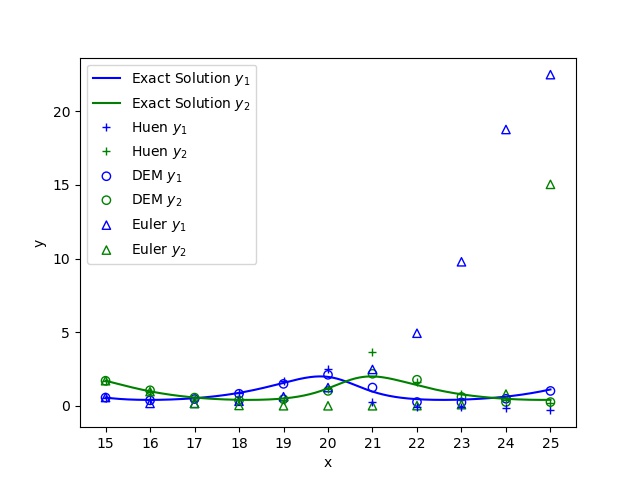}
		\caption{$h=1.0$ in [15,25]} 
	\end{subfigure}
	\caption{
		The evolutions  of the Euler method, the Huen method and DEM are displayed for noise-free measured data. The four different step sizes ($h=0.01, 0.1, 0.5$ and $1.0$)  are considered.
		} 
	\label{fig:5}
\end{figure}

Figure \ref{fig:4} shows the neural network ${\cal N}$ of DEM and the local truncation error function $R$ in both the training region $[0,15]$ and the testing region $(15,25]$.  The four different cases of $h=0.01, 0.1, 0.5, 1.0$ are displayed.  
It shows that DEM can accurately approximate the solution in the whole region  $(15,25]$ even for  $h=1.0$.
		In Figure \ref{fig:5}, we show the evolutions of the Euler method, the Huen method and DEM for comparison. 
		For the cases of $h=0.5$ and $h=1.0$, the results of the Euler method over $[0,15]$ are displayed in the left sub-figure. 
		For the case of $h = 0.5$, the solutions of the Euler method diverge from the exact value very much near $x=15$. When $h = 1.0$, the solution $y_2$ of the Euler method is close to $4$, but the solution of $y_1$ is near $0$, both of which are far away from the exact value. The Huen method has similar defects. When $h=1.0$, the solutions of the Huen method are close to $30000$, which is far away from the exact value. It can be observed that 		
		no matter the small $h$ or the big $h$, the results of DEM and the curve of the exact solution almost coincide. 

\subsection{Example 3}
Considering the Kepler problem :
\begin{equation}\label{eq:Example3_Kepler_eq}
		\left\{
		\begin{aligned} 
		\frac{d y_{1}}{d x} &=y_{3}, \\ 
		\frac{d y_{2}}{d x} &=y_{4}, \\ 
		\frac{d y_{3}}{d x} &=-\frac{y_{1}}{\left(y_{1}^{2}+y_{2}^{2}\right)^{3 / 2}}, \\ 
		\frac{d y_{4}}{d x} &=-\frac{y_{2}}{\left(y_{1}^{2}+y_{2}^{2}\right)^{3 / 2}}. 
		\end{aligned}
		\right.
\end{equation}

It describes the motion of the sun and a single planet which is a special case of the two-body problem. We denote the time by$x$. %
Let $(y_1(x),  y_2(x)) $ be the positions of the planet in rectangular coordinates centered at the sun and $y_3(x), y_4(x)$ be the velocity components in the $y_1$ and $y_2$ directions. The initial value  are $y(0)=[1,0,0,1]^T$, and the exact solution is $y(x)=[\cos (x), \sin (x),-\sin (x), \cos (x)]^{T}$. 

Similar to example 2, we use the uniform distribution to generate $1000$ noise-free data points and select the same values of the super parameters as in Example 1.  In Figure \ref{fig:6}, four different cases of $h$ are displayed. It can be observed that $\mathcal{N}(x_m,x_{m+1},\boldsymbol{y}_m)$ of DEM  well approximates $R(x_m,x_{m+1},\boldsymbol{y}_m,\boldsymbol{y}_{m+1})$ for all of $h$. 
We can also note that the approximations of DEM (circles in the figure) almost coincides with the curve of the exact solution in Figure \ref{fig:7}, even for $h=1.0$. 

\section{Conclusion}
		In this work, we proposed a Deep Euler Method with the idea of approximating the truncation error in the Euler method via deep learning. When deep neural network $\mathcal{N}$ is trained to approximate the local truncation error function with accuracy $\mathcal{O}(\eta)$, the global truncation error of DEM with the step size $h$ would be  $\mathcal{O}(\eta h)$ while the Euler method is only $\mathcal{O}(h)$. %
		Since $\eta$ can be small enough, it could  achieve high accuracy solutions even with a big step size ($h\geq 1$). %
		DEM significantly improves the accuracy of the Euler method and 
		reduces the constrain of the step size in the Euler method. 
		On the other hand, since 
		the training objective function of the deep neural network in DEM is always $\mathcal{O}(1)$, 
		the deep neural network can be easily trained and fast to converge, 
		even if only the simplest architecture of the fully connected network and only  a few training data are used. 
		Moreover, DEM shows good robustness with the noise of the measured data. 

\begin{figure}[H] 
	\begin{subfigure}{0.48\textwidth}
		\includegraphics[width=\linewidth]{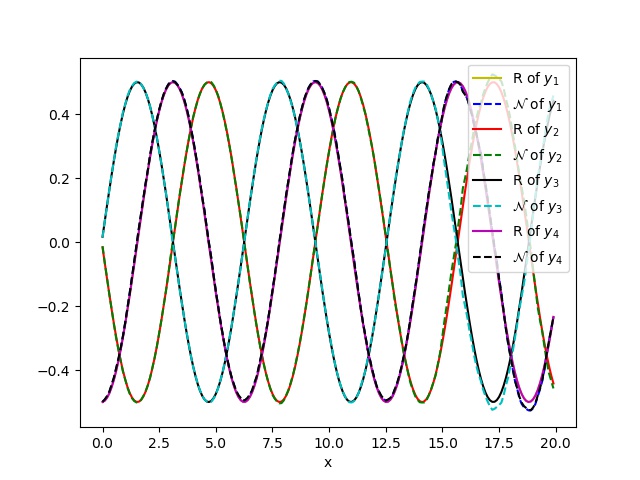}
		\caption{h=0.1} 
	\end{subfigure}\hspace*{\fill}
	\begin{subfigure}{0.48\textwidth}
		\includegraphics[width=\linewidth]{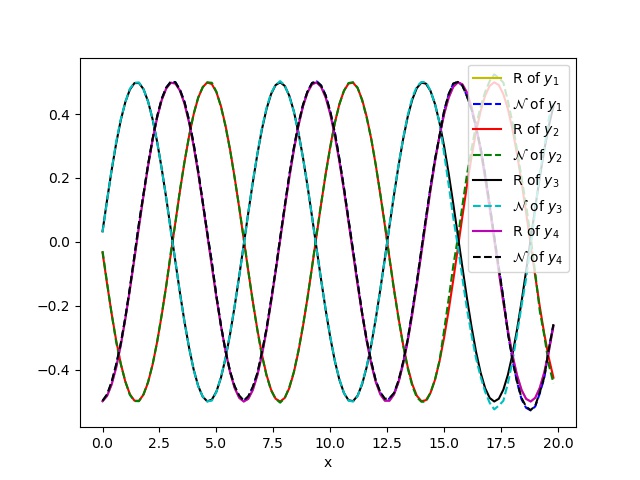}
		\caption{h=0.2} %
	\end{subfigure}
	
	\medskip
	\begin{subfigure}{0.48\textwidth}
		\includegraphics[width=\linewidth]{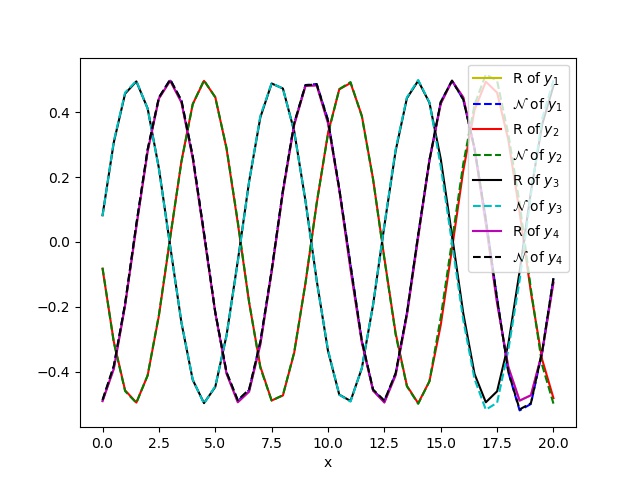}
		\caption{h=0.5} %
	\end{subfigure}\hspace*{\fill}
	\begin{subfigure}{0.48\textwidth}
		\includegraphics[width=\linewidth]{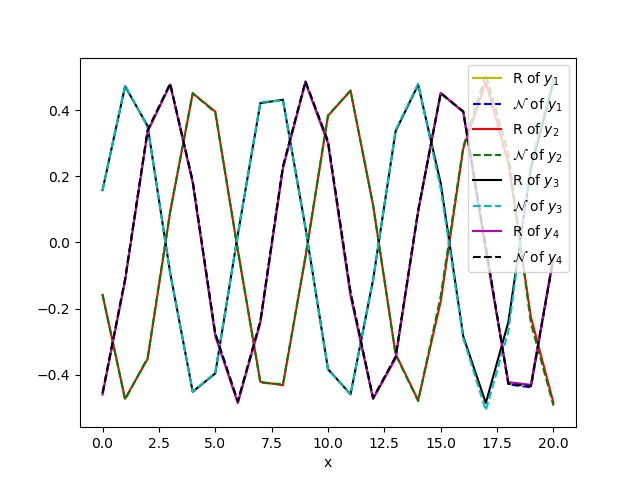}
		\caption{h=1.0} 
	\end{subfigure}
	\caption{The evolutions of the neural network $\mathcal{N}(x_m,x_{m+1},\boldsymbol{y}_m)$ and the local truncation error function $R(x_m,x_{m+1},\boldsymbol{y}_m,\boldsymbol{y}_{m+1})$ of the equation (\ref{eq:Example3_Kepler_eq}) for $h = 0.1,0.2,0.5,1.0$, respectively.}
	\label{fig:6}
\end{figure}

\begin{figure}[H] 
	\begin{subfigure}{0.4\textwidth}
		\includegraphics[width=\linewidth]{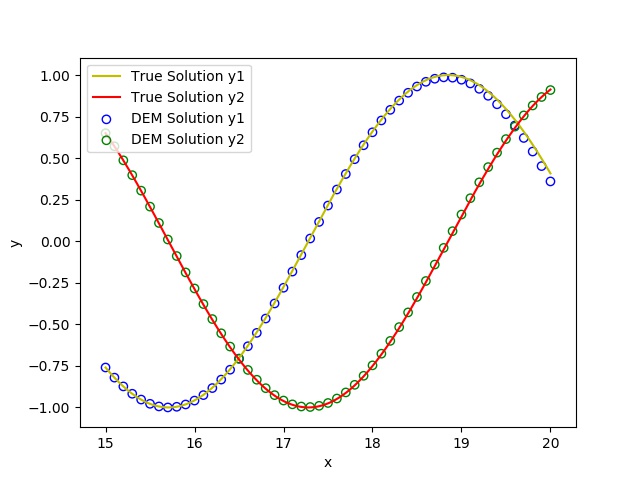}
		\caption{The components of $y_1,y_2$ with $h=0.1$} 
	\end{subfigure}\hspace*{\fill}
	\begin{subfigure}{0.4\textwidth}
		\includegraphics[width=\linewidth]{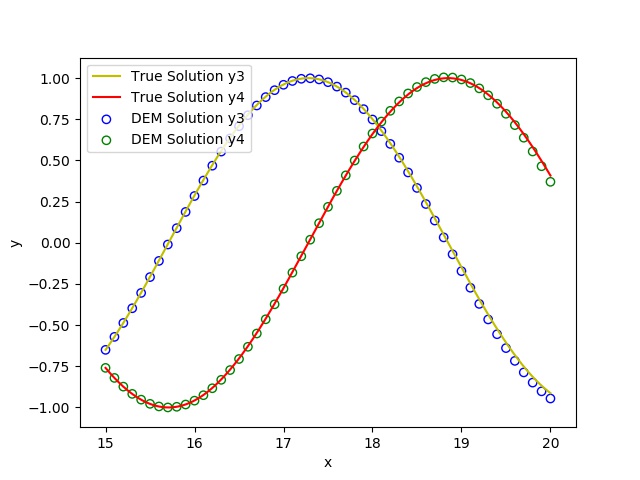}
			\caption{The components of $y_3,y_4$ with $h=0.1$} 
	\end{subfigure}
	
	\medskip
	\begin{subfigure}{0.4\textwidth}
		\includegraphics[width=\linewidth]{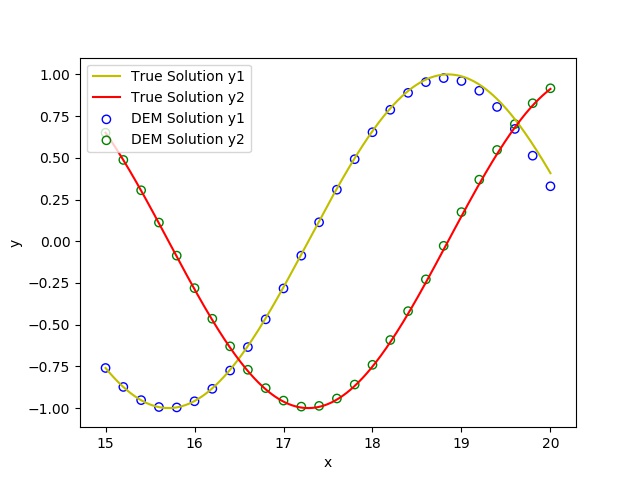}
			\caption{The components of $y_1,y_2$ with $h=0.2$} 
	\end{subfigure}\hspace*{\fill}
	\begin{subfigure}{0.4\textwidth}
		\includegraphics[width=\linewidth]{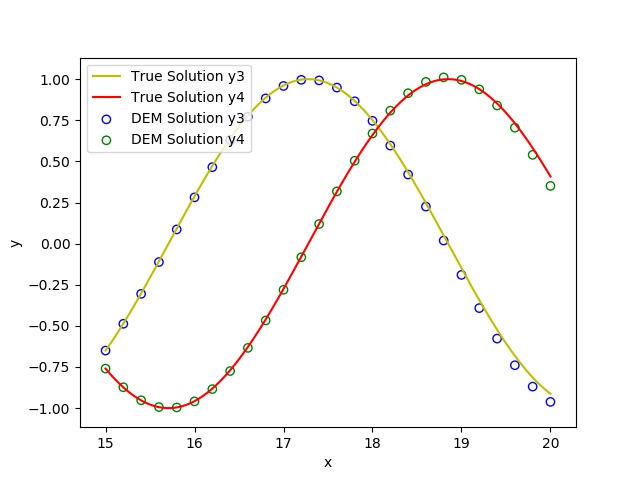}
						\caption{The components of $y_3,y_4$ with $h=0.2$} 
	\end{subfigure}
	
	\medskip
	\begin{subfigure}{0.4\textwidth}
		\includegraphics[width=\linewidth]{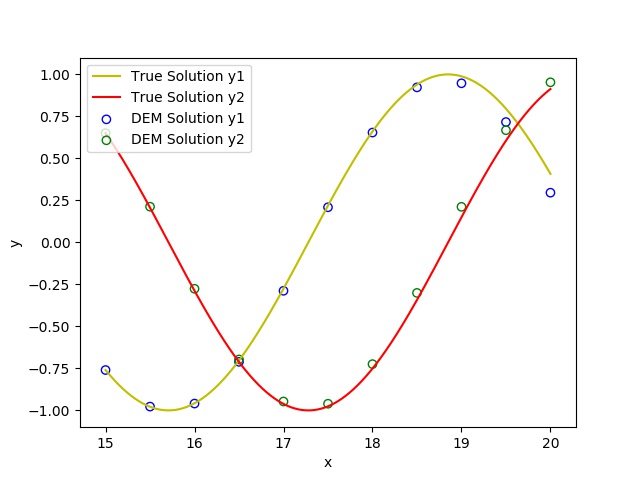}
			\caption{The components of $y_1,y_2$ with $h=0.5$} 
	\end{subfigure}\hspace*{\fill}
	\begin{subfigure}{0.4\textwidth}
		\includegraphics[width=\linewidth]{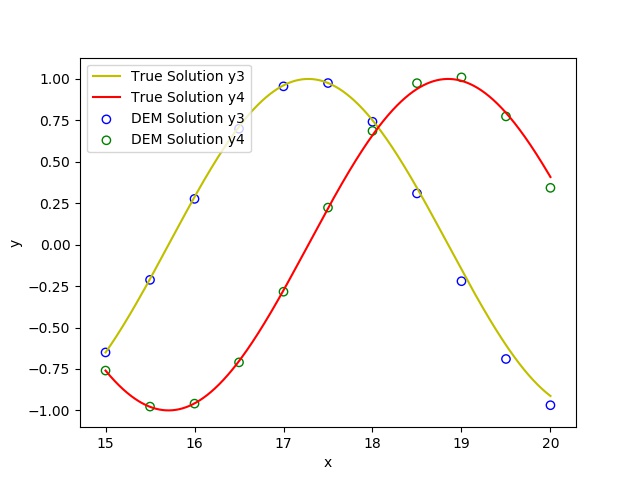}
						\caption{The components of $y_3,y_4$ with $h=0.5$} 
	\end{subfigure}
	
	\medskip
	\begin{subfigure}{0.4\textwidth}
		\includegraphics[width=\linewidth]{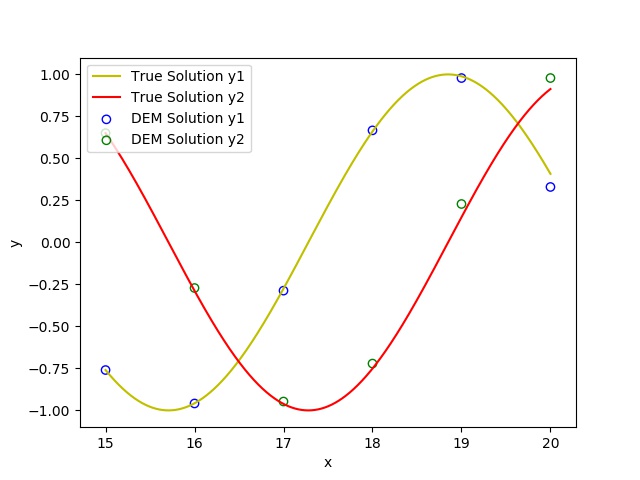}
			\caption{The components of $y_1,y_2$ with $h=1.0$} 
	\end{subfigure}\hspace*{\fill}
	\begin{subfigure}{0.4\textwidth}
		\includegraphics[width=\linewidth]{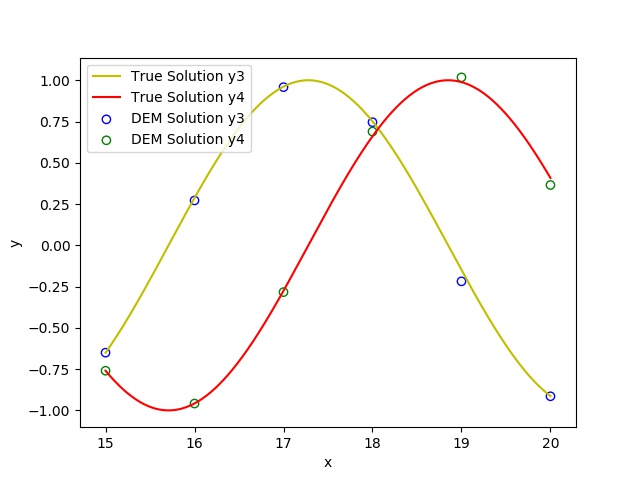}
						\caption{The components of $y_3,y_4$ with $h=1.0$} 
	\end{subfigure}			
	\caption{The exact solution and the %
		approximation of DEM of the equation (\ref{eq:Example3_Kepler_eq}) on region $(15,20]$ for %
		$h=0.1, 0.2, 0.5, 1.0$, respectively.}
	\label{fig:7}
\end{figure}

\printbibliography

\end{document}